\documentclass[12pt]{amsart}

\usepackage{amsmath,amssymb,txfonts}
\usepackage{amssymb}
\usepackage{amsxtra}
\usepackage{amsmath}
\usepackage{txfonts}
\def\Xint#1{\mathchoice
{\XXint\displaystyle\textstyle{#1}}%
{\XXint\textstyle\scriptstyle{#1}}%
{\XXint\scriptstyle\scriptscriptstyle{#1}}%
{\XXint\scriptscriptstyle\scriptscriptstyle{#1}}%
\!\int}
\def\XXint#1#2#3{{\setbox0=\hbox{$#1{#2#3}{\int}$ }
\vcenter{\hbox{$#2#3$ }}\kern-.6\wd0}}

\def\dashint{\Xint-}
\def\fint{\dashint}

      \newtheorem{theorem}{Theorem}[section]
      \newtheorem{remark}[theorem]{Remark}
      
      \newtheorem{definition}[theorem]{Definition}
      
      \newtheorem{lemma}[theorem]{Lemma}

      \newcommand{\ct}[1]{\langle {#1}\rangle \lower.3ex\hbox{$_{t}$}}
      \newcommand{\lt}[1]{[ {#1}] \lower.3ex\hbox{$_{t}$}}

\newcommand{\RR}{\mathbb{R}}

\begin{document}

\title[Singularities of Nonlinear Elliptic Systems]{Singularities of Nonlinear Elliptic Systems}
\author{David R. Adams}
\address{Department of Mathematics, University of Kentucky, Lexington, KY 40506-0027}
\email{dave@ms.uky.edu}

\author{Jie Xiao}
\address{Department of Mathematics and Statistics, Memorial University of Newfoundland, St. John's, NL A1C 5S7, Canada}
\email{jxiao@mun.ca}
\thanks{JX was supported in part by NSERC of Canada.}


\date{}
\dedicatory{Dedicated to N. G. Meyers on the occasion of his 80th
birthday}


\begin{abstract}
  Through Morrey's spaces (plus Zorko's spaces) and their potentials/capacities as well as Hausdorff contents/dimensions, this paper estimates the singular sets of nonlinear elliptic systems of the even-ordered Meyers-Elcrat type and a class of quadratic functionals inducing harmonic maps.
\end{abstract}
\maketitle

\tableofcontents

\section{Introduction}\label{s1}
\setcounter{equation}{0}

In \cite{M1938}, C. B. Morrey discovered a condition satisfied by the first derivatives of weak solutions to certain quasilinear second order systems of elliptic partial differential equations (pde) in domains (connected open sets) $\Omega\subseteq\mathbb R^n$ that implied everywhere $C^\alpha$ = H\"older continuity (of exponent $\alpha$) of the solutions throughout $\Omega$, when $n=2$. His condition -- now known as the ``Morrey condition" -- is:
\begin{equation}\label{e11}
\fint_{B_r(x_0)}\Big|\big(\frac{\partial}{\partial x}\big)^m u\Big|^p\le C r^{-\lambda}
\end{equation}
for all open balls $B_r(x_0)=\{x\in\mathbb R^n:\ |x-x_0|<r\}\subseteq\Omega$; $0<\lambda\le n, 1\le p<\infty, m\in\mathbb N$ (for derivatives of order $m$), and $C$ is a positive constant and $\fint_{E}$ stands for the integral average over $E$ with respect to the Lebesgue measure. Thus was born Morrey's Lemma:
$$
(\ref{e11})\Longrightarrow u\in C^\alpha\quad\hbox{for}\quad \alpha=m-\frac{\lambda}{p}>0.
$$
Notice that in the Sobolev theory with merely $(\partial/\partial x)^mu\in L^{p}(\Omega)$, i.e., $\lambda=n$, one generally needs $m-{n}/{p}>0$ to achieve H\"older continuity. Thus a significant gain is achieved from the Morrey condition. And we will henceforth say that a function $f$ is a Morrey class function on a domain $\Omega$ if it satisfies (\ref{e11}) with $f$ replacing $(\partial/\partial x)^mu$. Furthermore, we will say that these functions belong to the Morrey space $L^{p,\lambda}(\Omega)$.

Some what later, De Giorgi in \cite{DG1968}, gave an explicit example of a system of elliptic pde that could develop internal singularities provided the dimension of the underlying space exceeds two. This example, often quoted, shows that
\begin{equation}\label{e12}
u(x)=(u^1(x),...,u^n(x))=\frac{x}{|x|^\gamma}=\frac{(x_1,...,x_n)}{\big(\sum_{k=1}^n x_k^2\big)^{\gamma/2}}
\end{equation}
is a $W^{1,2}$ (Sobolev space)-solution of
\begin{equation}\label{e13}
-\big((a_{ij}^{kl}(x, u) u_{x_i}^{k}\big)_{x_j} = 0\ \ \forall\ \ l = 1,
\ldots, n
\end{equation}
(summation convention) with
$$
\gamma=\frac{n}{2}\Big(1-1/\sqrt{4(n-1)^2+1}\Big)
$$
and
$$
a_{ij}^{kl}(x, u) = \delta_{ij}\delta_{kl} +
  (c\,\delta_{ik} + d \,b_{ik}(x, u))(c\,\delta_{jl} + d\, b_{jl}(x, u))
$$
where $c,d$ are two positive constants; $c=n-2$, $d=n$ and
$$
b_{ik}(x,u) = \frac{x_{i}x_{k}}{|\,x|^2}
$$
in the De Giorgi case. Then, soon after, Giusti-Miranda \cite{GM1968} followed with the $u$ of (\ref{e12}) a solution of (\ref{e13}) with $\gamma=1, c=1, d=4/(n-2)$, and
$$
b_{ik}(x,u) = \frac{u^{i}u^{k}}{1 + |u|^2}.
$$
And more recently, Koshelev \cite{K1995} has refined the De Giorgi example by showing that again (\ref{e12}) solves (\ref{e13}) with
$$
\gamma=1,\quad c=(n-1)^{-1/2}\left(1+\frac{(n-2)^2}{n-1}\right)^{-1/4},\quad d=\frac{c+c^{-1}}{n-2}.
$$
Furthermore, Koshelev's example is extremal in a certain sense; see \cite[Chapter 8]{BF2002}.

In this paper, we  wish to revisit this question of the size of the singular set for such (higher order) systems, specifically the Meyers-Elcrat system \cite{ME1975} and then make some observations concerning some other nonlinear systems, e.g. the harmonic map system \cite{BF2002, LW2008}. But, the main point we wish to make here, is that a fundamental principal regarding the Morrey theory has gone unnoticed up to now: the Morrey condition can also be used to say something about weak solutions when one is operating below the continuity threshold, i.e., $m-{\lambda}/{p}\le 0$, $0<\lambda\le n$. Our results show that one can gain as much as $n-\lambda$ off the apriori dimension estimates of the singular sets in the Morrey case vs. the Sobolev case. And thus with coefficients of the pde satisfying additional regularity -- e.g. uniform continuity away from the singular set -- then one achieves the so-called partial regularity: the singular set is relatively closed in $\Omega$ and the solution is regular in the compliment (say $C^\alpha$ or even $C^\infty$, as in the harmonic map system case).

As mentioned, our main study will be the Meyers-Elcrat system of $2m$-th order quasilinear elliptic equations -- given below. However, a comment about our methods should be given here. The underlying Morrey theory need comes from a series of papers by the authors \cite{AX2004, AX2011a, AX2011b, AX2011c}, and in particular, from the estimates on the capacities associated with potentials of functions in the Morrey space $L^{p,\lambda}$, i.e., Riesz potentials
$I_\alpha f(x)=\int |x-y|^{\alpha-n}f(y)$, where generally $f$ has compact support, $0<\alpha<n$, $n\ge 3$, and the integral is taken with respect to the $n$-dimensional Lebesgue measure. This is a natural extension of the nonlinear potential theory of \cite{AH1996} where $I_\alpha f$ plays a central role but for $f\in L^p=L^{p,n}$.

\section{The results}\label{s2}
\setcounter{equation}{0}

One of the main reasons that makes the Meyers-Elcrat system distinctive is that every $W^{m,p}$-solution $u$ has a ``reverse H\"older" exponent $q>p$, i.e., $(\partial/\partial x)^m u\in L^{q}$ on $\Omega$. This idea originated from the $2^{nd}$ order case treated earlier by Meyers \cite{M1963}, but in this 1975 paper of Meyers-Elcrat, they rely on a device discovered by Ghering \cite{G1973} for derivatives of quasi-conformal maps, a device that has since been made into a force in regularity theory for nonlinear elliptic equations by Giaquinta-Giusti \cite{GG1982}. Normally, however, one can not expect to get H\"older continuity from reverse H\"older, though an increase in integrability exponent of solutions often helps, i.e., $u\in L^r$ on $\Omega$, for some $r>np/(n-mp)$ = Sobolev exponent, $mp<n$. Thus, one generally gets H\"older continuity of solutions only when the reverse H\"older exponent $q$ is sufficiently large and/or the Morrey exponent $\lambda$ is sufficiently small.

The Meyers-Elcrat system is:

\begin{equation}\label{e21}
  \sum_{|\gamma| \leq m} (-1)^{|\gamma|} \left(\frac{\partial}{\partial x}\right)^{\gamma} A_{\gamma}(x, D^{m}u) =
  0\quad\hbox{on}\quad\Omega,
\end{equation}
where
$$
A_{\gamma}: \Omega \times \mathbb R^{N} \rightarrow \mathbb R^N,\quad N=\sum_{k=1}^m n^k,
$$
is a Caratheodory function and
\begin{equation}\label{e22}
\begin{cases}
\sum_{|\gamma|\le m} A_{\gamma}(x, D^{m}u)
\left(\frac{\partial}{\partial x}\right)^{\gamma} u\geq a_{0}
\left|\left(\frac{\partial}{\partial x}\right)^{m}
u\right|^{p}, & \hbox{a.e.\ on}\ \Omega;\\
|A_{\gamma}(x, D^{m}u)| \leq M
\left|\left(\frac{\partial}{\partial x}\right)^{m} u\right|^{p -
1},\ \ |\gamma| \leq m, & \hbox{a.e.\ on}\ \Omega.
\end{cases}
\end{equation}
Here, $p\in (1,\infty)$, $a_0$ and $M$ are positive constants, and
$$
\begin{cases}
D^{m}u =\left\{ \left(\frac{\partial}{\partial x}\right)^{\gamma}u:
|\gamma| \leq m\right\};\\
\left(\frac{\partial}{\partial
x}\right)^{m} u=\left\{
  \left(\frac{\partial}{\partial x}\right)^{\gamma}u: |\gamma| = m\right\};\\
\left(\frac{\partial}{\partial x}\right)^{\gamma} = \left(\frac{\partial}{\partial x_1}\right)^{\gamma_1} \ldots \left(\frac{\partial}{\partial x_n}\right)^{\gamma_n};\\
\gamma =(\gamma_1, \ldots, \gamma_n)\in\mathbb N^n;\\
|\gamma| =\gamma_1+ \cdots + \gamma_n.
\end{cases}
$$
Our main result is:

\begin{theorem}\label{t1} Let $\Omega\subset\mathbb R^n$ be a bounded domain. If $u$ is a $W^{m,p}\cap L^q$-solution of (\ref{e21})-(\ref{e22}) with $q>np/(n-mp)$, then $|D^mu|$ belongs to $L^{p,\lambda}(\Omega)$ with $\lambda=(m+n/q)p<n$ and consequently the singular set
$$
\Sigma_{\hat{p}}(u,\Omega)=\mathsf{S}_{\hat{p}}(u,\Omega)\cup \mathsf{T}(u,\Omega)
$$
has Hausdorff dimension $\le np/q$. Here $\hat{p}$ equals $1$ or $p$, and
$$
\begin{cases}
\mathsf{S}_{\hat{p}}(u,\Omega)=\left\{x_0\in \Omega: \limsup_{r \to 0}
\fint_{B_r(x_0)}\big|u -\fint_{B_r(x_0)}u\big|^{\hat{p}}> 0 \right\};\\
\mathsf{T}(u,\Omega)=\left\{x_0\in\Omega: \sup_{r>0}\big|\fint_{B_r(x_0)\subseteq\Omega}u\big|=\infty\right\}.
\end{cases}
$$
Furthermore, when $q=\infty$, i.e., bounded solutions, then the singular sets have Hausdorff dimension zero, matching the examples of Giusti-Miranda and Koshelev where bounded isolated point singularities can occur. Further, if the coefficients are regular away from the singular set, then they are isolated points; see \cite[Chapter IX]{G1983}.
\end{theorem}

Next, we notice that our methods can be applied to getting estimates of the singular sets for a class of minima for certain quadratic functionals. These functionals with summation convention take the form
\begin{equation}\label{e23}
\mathcal{J}(u,\Omega)=\int_\Omega A^{kl}_{ij}(x,u) u^i_{x_k} u^j_{x_l}
\end{equation}
with
$$
\begin{cases}
{\hbox{symmetry}}:\ A^{kl}_{ij}(x,u)=A^{lk}_{ji}(x,u);\\
{\hbox{boundedness}}:\ |A^{kl}_{ij}(x,u)|\le M\ \hbox{for some constant}\ M>0;\\
{\hbox{ellipticity}}:\ A^{kl}_{ij}(x,u)\xi_k^i\xi_l^j\ge a_0|\xi|^2\ \hbox{for some constant}\ a_0>0;\\
{\hbox{H\"older coefficients}}:\ \frac{|A^{kl}_{ij}(x,z)-A^{kl}_{ij}(x',z')|}{|(x-x',z-z')|^\beta}\lesssim 1\ \hbox{for some constant}\ \beta\in (0,1);\\
{\hbox{splitting coefficients}}:\ A^{kl}_{ij}(x,u)=g_{ij}(x,u)G^{kl}(x),
\end{cases}
$$
where $\mathsf{X}\lesssim\mathsf{Y}$ stands for $\mathsf{X}\le c\mathsf{Y}$ for a constant $c>0$. Our result is:

\begin{theorem}\label{t2} Let $\Omega\subset\mathbb R^n$ be a bounded domain. If $u$ is a bounded $W^{1,2}$-minimizer for the functional $\mathcal{J}(u,\Omega)$ with $A^{kl}_{ij}$ satisfying the above, then $|\nabla u|\in L^{2,2}(B)$ for any $B\Subset\Omega$, and consequently all local singular sets (in the sense of \cite{HL1990}, say) of such minimizers have Hausdorff dimension zero. Thus if $\mathcal{J}(\cdot,\cdot)$ is the energy functional for harmonic maps into a compact Riemannian manifold with smooth coefficients, then the solutions (the minimizing harmonic maps) have only sets of isolated points as local singular sets.
\end{theorem}

Nevertheless, the result of Theorem \ref{t2} applies only to the so-called local or isolated singular sets because one cannot ``localize" a singular set of a minimizing harmonic map that extend to the boundary of $\Omega$ (as in \cite{HL1990} -- more about this below). This applies to the singularities that arise, for example, with energy minimizing maps that are independent of a variable:
$u(x,y)={x}/{|x|}$ with $(x,y)\in\mathbb R^3\times\mathbb R^{n-3}$. Such a $u$ is a minimizer for the harmonic map system with pde
$$
-\Delta_g u=A(u)\langle \nabla u,\nabla u\rangle
$$
with quadratic growth on the right side; see \cite{LW2008}. Now the fact that such solutions satisfy
$$
\sup_{B_r(x_0)\Subset\Omega}r^2\fint_{B_r(x_0)}\Big|\big(\frac{\partial}{\partial x}u\big)\Big|^2<\infty
$$
is well known -- it just follows from the ``monotone inequality" for minimizing harmonic maps -- see \cite[Chapter IX]{G1983} or \cite{LW2008} -- because $u$ is in fact a $W^{1,2}$-solution. Thus one can achieve a singular set as large as dimension $n-3$, the maximum allowable for minimizing harmonic maps. Hence for minimizing harmonic maps, all local or isolated singular sets consist of just isolated points, as in \cite{HL1990}.

Our final result again relates to minimizing harmonic maps and their singular sets.

\begin{theorem}\label{t3} Let $\mathbb B^n$ and $\mathbb S^{m-1}$ are the unit ball of $\mathbb R^n$ and the unit sphere of $\mathbb R^m$. If $u=(u^1,...,u^m)$ is a minimizing harmonic map from $\mathbb B^n$ to $\mathbb S^{m-1}$ and $\hbox{sing}(u,\mathbb B^n)$ is the set of all discontinuous points of $u$, then
\begin{equation}\label{eqF}
\hbox{sing}(u,\mathbb B^n)=\left\{x\in\mathbb B^n:\ I_1\Big(\Big|\big(\frac{\partial}{\partial x}\big)u\Big|\Big)(x)=\infty\right\},
\end{equation}
  namely, we are saying that bounded singularities of $u$ correspond to unbounded discontinuities of the $1$-Riesz potential of the Hilbert-Schmidt norm $|\big({\partial}/{\partial x}\big)u|$ of $\big({\partial}/{\partial x}\big)u$ determined by:
$$
\Big|\big(\frac{\partial}{\partial x}\big)u\Big|^2=\sum_{i=1}^m\sum_{j=1}^n\Big|\big(\frac{\partial}{\partial x_j}\big)u^i\Big|^2.
$$
\end{theorem}

These last results on minimizing harmonic maps are in sharp contrast to the singular set results that can occur in the Yamabe problem: $-\Delta u=u^\frac{n+2}{n-2}$, $u\ge 0$, in $\Omega\subset\mathbb R^n$, $n\ge 3$. Here, it has been shown that the largest singular set one can have here, has dimension $(n-2)/2$ and this can be realized. And on the other hand one can also have local singular sets of the Cantor type along a line in the complement of where the solution is regular; see \cite{SY1988}. And these sets can have dimension positive and arbitrarily small! This can not happen for minimizing harmonic maps.

\section{The proofs}\label{s3}
\setcounter{equation}{0}

\subsection{Three definitions} We need concepts of the so-called Zorko space, Hausdorff capacity/dimension, and Morrey capacity.

\begin{definition}\label{d1} Given a domain $\Omega\subseteq\mathbb R^n$ and $1<p<\infty, 0<\lambda\le n$, each Morrey  space $L^{p,\lambda}$ on $\Omega$ is equipped with the following norm
$$
\|f\|_{L^{p,\lambda}(\Omega)}=\left(\sup_{B_r(x_0)\subseteq\Omega}r^\lambda\fint_{B_r(x_0)}|f|^p\right)^\frac1p.
$$
We say $f\in L^{p,\lambda}_0(\Omega)$ (the Zorko space
\cite{Z1986}) whenever $f$ can be approximated by
$C_0^1(\Omega)$-functions in the norm $\|\cdot\|_{L^{p,\lambda}(\Omega)}$.
\end{definition}

Each Morrey space has its own capacity.

\begin{definition}\label{d2} For a domain $\Omega\subseteq\mathbb R^n$, $1<p<\infty, 0<\lambda\le n$, $0<\alpha<n$ and $E\subseteq\Omega$, let
$$
C_\alpha(E;L^{p,\lambda}(\Omega))=\inf\{\|f\|_{L^{p,\lambda}(\Omega)}^p:\
0\le f\in L^{p,\lambda}(\Omega)\ \ \&\ \ I_\alpha f\ge 1_E\},
$$
where $1_E$ stands for the characteristic function of $E$.
\end{definition}

According to \cite[Theorem 5.3]{AX2004}, we know that if $B_r(x_0)\subseteq\Omega$ converges to $x_0$ then
$$
C_\alpha(B_r(x_0);L^{p,\lambda}(\Omega))\approx\begin{cases}
                              r^{\lambda-\alpha p},  &   1<p<\lambda/\alpha; \\
                              (-\ln r)^{-p},  &   1<p=\lambda/\alpha. \\
                          \end{cases}
$$
Here and later on, ${\mathsf X}\approx{\mathsf Y}$ represents that there exists a constant
$c>0$ such that $c^{-1}{\mathsf Y}\le{\mathsf X}\le c{\mathsf Y}$.

\begin{definition}\label{d3} The classical $(0,n]\ni d$-dimensional Hausdorff capacity of a set $E\subset\mathbb R^n$ is defined via:
$$
\Lambda_d^{(\infty)}(E)=\inf\sum_j r_j^d,
$$
where the infimum is taken over all countable coverings of $E$ by balls $B_{r_j}(\cdot)$. Moreover, the
Hausdorff dimension of $E$ is decided by
$$
\hbox{dim}_H(E)=\inf\{d:\ \Lambda_d^{(\infty)}(E)=0\}.
$$
\end{definition}

\subsection{Two lemmas} Our first lemma indicates that each Morrey space is actually
embedded into the intersection of a family of the Zorko spaces:
\begin{lemma}\label{l1} Let $\Omega\subset\mathbb R^n$ be a bounded domain. Then
\begin{equation}\label{e32}
L^{p,\lambda}_0(\Omega)\subset L^{p,\lambda}(\Omega)\subset\cap_{\lambda<\mu<
n}L^{p,\mu}_0(\Omega).
\end{equation}
\end{lemma}
\begin{proof} The first inclusion of (\ref{e32}) follows from Definition \ref{d1}. To validate the second inclusion in (\ref{e32}), via setting $f=0$ outside $\Omega$, we may assume that $f$ is in $L^{p,\lambda}(\mathbb R^n)$ = the Morrey space $L^{p,\lambda}(\Omega)$ with $\Omega$ replaced by $\mathbb R^n$, and then $f_\epsilon$ is the $\epsilon$-mollifier of $f$, i.e.,
$$
f_\epsilon(x)=\phi_\epsilon\ast f(x)=\int_{\mathbb R^n}\epsilon^{-n}\phi(\epsilon^{-1}y)f(x-y)
$$
where
$$
\phi\in C^\infty_0(\mathbb R^n);\ 0\le\phi\le 1;\ \int_{\mathbb R^n}\phi=1;\ \phi_\epsilon(x)=\epsilon^{-n}\phi(x/\epsilon).
$$
For $\lambda<\mu<n$ let $q=(n-\lambda)/(n-\mu)$ and $q'=q/(q-1)$. Then
$$
r^{\mu-n}\int_{B_r(x_0)}|f-f_\epsilon|^p=\left(r^{\lambda-n}\int_{B_r(x_0)}|f-f_\epsilon|^p\right)^\frac1q\left(\int_{B_r(x_0)}|f-f_\epsilon|^p\right)^\frac{1}{q'}
$$
and hence
$$
\|f-f_\epsilon\|_{L^{p,\mu}(\Omega)}\le\|f-f_\epsilon\|_{L^{p,\lambda}(\Omega)}^\frac1q\|f-f_\epsilon\|_{L^{p}(B_R(0))}^\frac1{q'}
$$
for some large finite $R>0$. Note that $f_\epsilon\to f$ in
$L^p_{loc}$ but at best $f_\epsilon$ is bounded in
$L^{p,\mu}$; see also Zorko \cite{Z1986}. Therefore, $f\in
L^{p,\mu}_0(\Omega)$.
\end{proof}

The analysis on Page 1649 of \cite{AX2004} gives that if
$1<p<\lambda/\alpha$ and $E\subset B_R(x_0)\subset\Omega$ then
\begin{equation*}\label{e1}
\Lambda_{n}^{(\infty)}(E)^\frac{\lambda-\alpha p}{\lambda}\lesssim
\frac{C_\alpha(E; L^{p,\lambda}(\Omega))}{R^\frac{(\lambda-\alpha p)(\lambda-n)}{\lambda}}\ \ \&\ \ C_\alpha(E; L^{p,\lambda}(\Omega))\lesssim\Lambda_{\lambda-\alpha
p}^{(\infty)}(E).
\end{equation*}
Geometrically speaking, the last estimates are rough isocapacitary inequalities for the Morrey capacity and the Hausdorff capacity. But, they can be improved to the following Morrey-Hausdorff isocapacitary inequalities extending the well-known result for $\lambda=n$; see also \cite{AH1996}.

\begin{lemma}\label{l2} Let $\Omega\subset\mathbb R^n$ be a bounded domain, $0<\alpha,\lambda<n$, $0\le\lambda-\alpha p<d\le n$ and $E\subseteq\Omega$.

\item{\rm(i)} If $1<p<\lambda/\alpha$ and $0<q<dp/(\lambda-\alpha p)$, then
$$
\Lambda_{d}^{(\infty)}(E)\lesssim C_\alpha\big(E;
L^{p,\lambda}(\Omega)\big)^\frac qp.
$$
\item{\rm(ii)} If $1<p=\lambda/\alpha$ and $0<q\le 1$, then there is a constant
$c>0$ such that
$$
\Lambda_{d}^{(\infty)}(E)\lesssim\exp\Big(-c C_\alpha\big(E;
L^{p,\lambda}(\Omega)\big)^\frac qp\Big)\quad\forall\ d\in (0,n].
$$
\end{lemma}
\begin{proof} On the one hand, suppose $\nu$ is a non-negative Borel measure on $\mathbb R^n$ obeying
\begin{equation*}\label{eqMeaL}
   \sup_{(r,x_0)\in (0,\infty)\times\mathbb R^n}\frac{\nu(B_r(x_0))}{r^d}<\infty\ \ \hbox{for}\ \
n\ge d >\lambda - \alpha \,p\ge 0.
\end{equation*}
According to \cite[Theorem 3.1]{AX2011b} (cf. \cite{AX2011a}), we have:

\noindent {\rm(i)} If $1<p<\frac{\lambda}{\alpha}$ and
$0<\lambda<n$, then
\begin{equation*}\label{eqTr32aL}
\sup_{\|f\|_{L^{p,\lambda}(\Omega)}\le1}\int_{\Omega}
|I_{\alpha}f|^{q} \, d\nu<\infty\quad\hbox{for}\quad
q<\dfrac{dp}{\lambda -\alpha p}
\end{equation*}
and
\begin{equation*}\label{eqTrIm32aL}
\sup_{\|f\|_{L^{p,\lambda}(\Omega)}\le1}\int_{\Omega}
\frac{|I_{\alpha}f|^{\tilde{p}}}{[\ln(1+|I_\alpha f|)]^\gamma} \,
d\nu<\infty\quad\hbox{for}\quad\tilde{p}=\dfrac{dp}{\lambda -\alpha
p}\ \&\ \gamma>2.
\end{equation*}

\noindent {\rm(ii)} If $1<p=\frac{\lambda}{\alpha}$ and
$0<\lambda\le n$, then there exists a constant $c>0$ such that
\begin{equation*}\label{eqTr2aa}
\sup_{\|f\|_{L^{p,\lambda}(\Omega)}\le 1}\int_{\Omega}
\exp\big(c|I_{\alpha}f|^{q}\big) \, d\nu<\infty
\end{equation*}
holds for $(\lambda,q)\in (0,n)\times (0,1]$ or
$(\lambda,q)=\Big(n,\frac{n}{n-1}\Big)$.

On the other hand, \cite[Corollary]{Ad} tells us that
under $d\in (0,n]$, one has
$$
\Lambda_{d}^{(\infty)}(E)\approx \sup_{\nu}\nu(E),
$$
where the ``sup" is taken over all non-negative Borel measures $\nu$
on $\mathbb R^n$ with
$$
\sup_{(r,x_0)\in (0,\infty)\times\RR^N}\frac{\nu(B_r(x_0))}{r^d}<\infty.
$$

So, the above-recalled facts, plus the definition of
$C_\alpha(E;L^{p,\lambda}(\Omega))$, derive the iso-capacitary
estimates in Lemma \ref{l2}.
\end{proof}

\subsection{One more theorem}  The following singularity result for the Morrey potentials $I_\alpha L^{p,\lambda}(\Omega)$ will be used later on.

\begin{theorem}\label{t4} Let $\Omega\subset\mathbb R^n$ be a bounded domain and $f\in L^{p,\lambda}(\Omega)$.

\item{\rm(i)} If $1<p<\lambda/\alpha<\mu/\alpha\le n/\alpha$, then
$$
C_\alpha\big(\Sigma_1(I_\alpha f,\Omega);
L^{p,\mu}(\Omega)\big)=0\quad\&\quad
\hbox{dim}_H\big(\Sigma_1(I_\alpha f,\Omega)\big)\le
\lambda-\alpha p
$$

\item{\rm(ii)} If $1<p=\lambda/\alpha<\mu/\alpha\le n/\alpha$, then
$$
C_\alpha\big(\Sigma_1(I_\alpha f,\Omega);
L^{p,\mu}(\Omega)\big)=0\quad\&\quad
\hbox{dim}_H\big(\Sigma_1(I_\alpha f,\Omega)\big)=0.
$$
\end{theorem}
\begin{proof} First of all, for $\epsilon>0$ let $f_\epsilon=\phi_\epsilon\ast f$ be of the $\epsilon$-mollifier of $f\in L^{p,\lambda}(\Omega)$ and $\mathcal{M}$ denote the Hardy-Littlewood maximal operator.

Next, let us treat $\mathsf{S}_1(I_\alpha f,\Omega)$. For $t>0$ set
$$
\mathsf{S}_1(I_\alpha f,\Omega,t)=\left\{x_0\in\Omega:\ \limsup_{r\to
0}\fint_{B_r(x_0)}\Big|I_\alpha f-\fint_{B_r(x_0)}I_\alpha
f\Big|>t\right\}.
$$
By Lemma \ref{l1}, we see $f\in L^{p,\mu}_0(\Omega)$ and then
\begin{eqnarray*}
t&\le&\limsup_{r\to 0}\fint_{B_r(x_0)}\big|I_\alpha f-\fint_{B_r(x_0)}I_\alpha f\big|\\
&\lesssim&\limsup_{r\to 0}\fint_{B_r(x_0)}\big|I_\alpha
(f-f_\epsilon)-\fint_{B_r(x_0)}I_\alpha(f-f_\epsilon)\big|\\
&&+\ \limsup_{r\to 0}\fint_{B_r(x_0)}\big|I_\alpha
(f_\epsilon)-\fint_{B_r(x_0)}I_\alpha(f_\epsilon)\big|\\
&\lesssim& \mathcal{M}\big(I_\alpha (|f-f_\epsilon|)\big)(x_0)\\
&\lesssim&I_\alpha\big(\mathcal{M}(|f-f_\epsilon|)\big)(x_0).
\end{eqnarray*}
By the definition of $C_\alpha\big(\cdot;L^{p,\mu}(\Omega)\big)$ and
the boundedness of $\mathcal{M}$ on $L^{p,\mu}$ (cf. \cite{CF1988}) with $p>1$
and $\mu>\lambda$, we get
$$
C_\alpha\big(\mathsf{S}_1(I_\alpha
f,\Omega,t);L^{p,\mu}(\Omega)\big)\lesssim
t^{-p}\|f-f_\epsilon\|_{L^{p,\mu}(\Omega)}^{p}\quad\forall\quad
\epsilon>0,
$$
whence finding (via letting $\epsilon\to 0$)
$$
C_\alpha\big(\mathsf{S}_1(I_\alpha
f,\Omega,t);L^{{p},\mu}(\Omega)\big)=0.
$$
Since $t>0$ is arbitrary, we obtain
$$
C_\alpha\big(\mathsf{S}_1(I_\alpha
f,\Omega);L^{{p},\mu}(\Omega)\big)=0.
$$
This, along with Lemma \ref{l2}(i), deduces
\begin{equation}\label{e+}
\Lambda_d^{(\infty)}\big(\mathsf{S}_1(I_\alpha
f,\Omega)\big)=0\quad\forall\quad d>\mu-\alpha p\quad\&\quad
0<q<\frac{p(\mu-\alpha p)}{d}.
\end{equation}
As a result, letting $\mu\to \lambda$, we find
$\hbox{dim}_H\big(\mathsf{S}(I_\alpha f,\Omega)\big)\le\lambda-\alpha p.$
In the last estimate, we have used $\lambda>\alpha p$. Nevertheless,
when $\lambda=\alpha p$, we still have (\ref{e+}) with
$\mu>\lambda$, and thereby reaching
$\hbox{dim}_H\big(\mathsf{S}(I_\alpha f,\Omega)\big)=0.$

Thirdly, we handle the case for $\mathsf{T}$. Set $d(x_0,\partial\Omega)$ be the distance of $x_0\in\Omega$ to the boundary $\partial\Omega$ of $\Omega$ and
$$
\mathsf{T}(I_\alpha f,\Omega,t)=\left\{x_0\in\Omega: \sup_{0<r<d(x_0,\partial\Omega)}\big|\fint_{B_r(x_0)}I_\alpha f\big|>t\right\}.
$$
Then by Lemma \ref{l1}, we get $f\in L^{p,\mu}_0(\Omega)\subset L^{p,\mu}(\Omega)$ and
\begin{eqnarray*}
t&<&\sup_{0<r<d(x_0,\partial\Omega)}\big|\fint_{B_r(x_0)}I_\alpha f\big|\\
&\lesssim & \sup_{0<r<d(x_0,\partial\Omega)}\fint_{B_r(x_0)}|I_\alpha(f-f_\epsilon)|+\sup_{0<r<d(x_0,\partial\Omega)} \fint_{B_r(x_0)}|I_\alpha(f_\epsilon)|\\
&\lesssim& \mathcal{M}(I_\alpha(|f-f_\epsilon|))(x_0)+\mathcal{M}(I_\alpha(f_\epsilon))(x_0)\\
&\lesssim& I_\alpha (\mathcal{M}(|f-f_\epsilon|))(x_0)+I_\alpha(\mathcal{M}(f_\epsilon))(x_0).
\end{eqnarray*}
Consequently, there is a constant $c>0$ such that
$$
\mathsf{T}(I_\alpha f,\Omega,t)\subseteq\mathsf{T}_1(I_\alpha f,\Omega,t)\cup\mathsf{T}_2(I_\alpha f,\Omega,t),
$$
where
$$
\begin{cases}
\mathsf{T}_1(I_\alpha f,\Omega,t)=\{x_0\in\Omega: I_\alpha (\mathcal{M}(|f-f_\epsilon|))(x_0)\ge\frac{ct}{2}\}\\
\mathsf{T}_2(I_\alpha f,\Omega,t)=\{x_0\in\Omega: I_\alpha(\mathcal{M}(f_\epsilon))(x_0)\ge\frac{ct}{2}\}.
\end{cases}
$$
A combined use of the definition of $C_\alpha(\cdot;L^{p,\mu}(\Omega))$, the boundedness of $\mathcal{M}$ on $L^{p,\mu}$ and $L^{p,\lambda}$ and the easily-verified uniform boundedness of $f\mapsto f_\epsilon$ on $L^{p,\mu}(\Omega)$ (cf. \cite{Z1986}) gives
\begin{eqnarray*}
&&C_\alpha(\mathsf{T}(I_\alpha f,\Omega,t); L^{p,\mu}(\Omega))\\
&&\lesssim C_\alpha(\mathsf{T}_1(I_\alpha f,\Omega,t); L^{p,\mu}(\Omega))+C_\alpha(\mathsf{T}_2(I_\alpha f,\Omega,t); L^{p,\mu}(\Omega))\\
&&\lesssim t^{-p}\|f-f_\epsilon\|^p_{L^{p,\mu}(\Omega)}+t^{-p}\|f_\epsilon\|_{L^{p,\mu}(\Omega)}^p\\
&&\lesssim t^{-p}\|f-f_\epsilon\|^p_{L^{p,\mu}(\Omega)}+t^{-p}\|f\|_{L^{p,\mu}(\Omega)}^p.
\end{eqnarray*}
Since $\lim_{\epsilon\to 0}\|f-f_\epsilon\|^p_{L^{p,\mu}(\Omega)}=0$, letting $\epsilon\to 0$ and then $t\to\infty$, one derives
$$
C_\alpha(\mathsf{T}(I_\alpha f,\Omega);L^{p,\mu}(\Omega))=0.
$$
This plus Lemma \ref{l2} yields
$$
\Lambda_d^{(\infty)}(\mathsf{T}(I_\alpha f,\Omega))=0\quad\forall\quad d>\mu-\alpha p>\lambda-\alpha p,
$$
whence giving
$$
\hbox{dim}_H(\mathsf{T}(I_\alpha f,\Omega))\le\lambda-\alpha p.
$$

Now, the above estimates yield the desired results for $\Sigma_1=\mathsf{S}_1\cup\mathsf{T}$.
\end{proof}

\subsection{Proof of Theorem \ref{t1}} The part on $q=\infty$ follows readily from the argument for the case $np/(n-mp)<q<\infty$. So, it is enough to handle this last case.

The result $|D^m u|\in L^{p,\lambda}(\Omega)$ with $\lambda=(m+n/q)p<n$ follows from the estimate below:
\begin{equation}\label{e33}
r^{mp}\int_{B_{r/2}(x_0)} \left|\left(\frac{\partial}{\partial
x}\right)^{m} u\right|^{p}\,\lesssim\int_{B_r(x_0)}|u|^p \ \
\forall\ \ x_0 \in \Omega\ \ \&\  0<r<d(x_0,\partial\Omega).
\end{equation}
To verify (\ref{e33}), we just use the test function $\varphi =
\eta^{mp}u$, where $\eta(x)=\psi\Big(\frac{x - x_0}{r}\Big)$ for which
$$
\psi\in C_0^\infty(\mathbb R^n)\quad\&\quad \psi(x)=\begin{cases} 1, & x\in B_{r/2}(x_0);\\
0, & x\in \mathbb R^n\setminus B_r(x_0).
\end{cases}
$$
This then gives

\begin{eqnarray*}
\int \left|\left(\frac{\partial}{\partial
x}\right)^{m}u\right|^{p}\eta^{mp}
   &\lesssim& \int \left|\left(\frac{\partial}{\partial
x}\right)^{m}u\right|^{p-1}\eta^{m(p-1)} \,
\left|\left(\frac{\partial}{\partial x}\right)^{m -
k}u\right|r^{-(m-k)}\\
&&+\ \int \left|\left(\frac{\partial}{\partial
x}\right)^{m}u\right|^{p-1}\eta^{m(p-1)}\left|\left(\frac{\partial}{\partial
x}\right)^{j - l}u\right|r^{-(j-l)}
\end{eqnarray*}
for $0 < k \leq m$ and $j < m$ with $0 \leq l \leq j$.  Then via the Young
inequality
$$
ab\le \frac{\epsilon a^{\theta}}{\theta}+\frac{\epsilon^{\frac{1}{1-\theta}}b^{\theta'}}
{\theta'}\quad{\forall}\quad a,\ b,\
\epsilon>0,\ \theta>1,\ \theta'=\frac{\theta}{\theta-1},
$$
we get
\begin{eqnarray*}
  \int \left|\left(\frac{\partial}{\partial x}\right)^{m}u\right|^p \eta^{mp}
  &\leq&\int_{B_r(x_0)}\left|\left(\frac{\partial}{\partial
x}\right)^{m-k}u\right|^{p} r^{-(m-k)p}\\
&&+\ \int_{B_r(x_0)}\left|\left(\frac{\partial}{\partial
x}\right)^{j - l}u\right|^{p}r^{-(j-l)p}.
\end{eqnarray*}
Now, applying the Gagliardo-Nirenberg inequality (see e.g.
\cite{F1969}) gives
$$
\int \left|\left(\frac{\partial}{\partial x}\right)^{m}u\right|^{p}
\eta^{mp} \lesssim r^{-mp}\int_{B_r(x_0)}|u|^{p}.
$$

Next, we prove
\begin{equation}\label{e34}
\hbox{dim}_H\big(\Sigma_{\hat{p}}(u,\Omega)\big)\le np/q.
\end{equation}
To reach (\ref{e34}), let $f=|D^m u|$ and consider two cases below.

{\it Case 1}:\ $\hat{p}=p$. Firstly, we establish the following estimate for $\mu\in (\lambda,n]$ and $t>0$:
\begin{equation}\label{e35}
  C_m\big(\{x_0\in\Omega;\big(\mathcal{M}(I_m f)^p(x_0)\big)^{1/p}> t\}; \;L^{p,\mu}(\Omega)\big) \:\lesssim t^{-p}\|f\|_{L^{p,\mu}(\Omega)}^{p}.
  \end{equation}
In fact, observe that
$$
\big(\mathcal{M}(I_{m}f)^p\big)^{1/p} \lesssim I_m (\mathcal{M}f^{p})^{1/p}
$$
 and so that the left side of (\ref{e35}) does not exceed $t^{-p}\|(\mathcal{M}f^{p})^{1/p}\|_{L^{p, \mu}(\Omega)}^{p}$
by definition of the Morrey capacity.  But clearly
$$
  \left( r^\mu \fint_{B_r(x_0)} \mathcal{M} f^p \right)
  \lesssim
  \left( r^{\mu(p+\epsilon)/p} \fint_{B_r(x_0)} (\mathcal{M}f)^{\frac{p + \epsilon}{p}}\right)^{\frac{p}{p + \epsilon}}
$$
holds for small number $\epsilon>0$. So
$$
\|\mathcal{M} f^{p} \|_{L^{p,\mu}(\Omega)} \lesssim \|\mathcal{M} f^{p}
\|_{L^{(p+\epsilon)/p, \mu(p+\epsilon)/p}(\Omega)}
\lesssim\|f\|_{L^{p+\epsilon, \mu(p + \epsilon)/p}(\Omega)}
$$
follows from \cite{CF1988}: the maximal function is a bounded operator
on the Morrey spaces $L^{(p+\epsilon)/p, \mu(p+\epsilon)/p}$. But
applying the reversed H\"older estimates for $f$ from \cite{ME1975}, we get
$$
\|f\|_{L^{p+\epsilon, \mu(p + \epsilon)/p}(\Omega)}
\lesssim\|f\|_{L^{p, \mu}(\Omega)}.
$$
Thus the desired result (\ref{e35}) follows.

Secondly, we need the fact that any $h\in C_0^\infty(\mathbb R^n)$ can be represented as
(\cite[Lemma 2]{AAnn}):
\begin{equation}\label{e36}
h(x)=\begin{cases}
(-1)^\frac{k}{2}\Big(\frac{\omega_{n-1}\beta_0}{n}\Big)^\frac{2-n}{n}\int_{\mathbb
R^n}\frac{\nabla^k h(y)}{|x-y|^{n-k}}, & k=2,4, 6,...\\
(-1)^\frac{k-1}{2}\Big(\frac{\omega_{n-1}\beta_0}{n}\Big)^\frac{2-n}{n}\int_{\mathbb
R^n}\frac{(x-y)\cdot\nabla^k h(y)}{|x-y|^{n-k+1}}, & k=1,3,5,...,
\end{cases}
\end{equation}
where $\omega_{n-1}=2\pi^{n/2}/\Gamma(n/2)$ is the volume of the
boundary $\mathbb S^{n-1}$ of the unit ball $\mathbb B^n$ of
$\mathbb R^n$, $\Gamma(\cdot)$ is the usual Gamma function,
$$
\nabla^k h=\begin{cases}
(-\Delta)^\frac{k}{2}h, & k=2,4,6,...\\
\nabla(-\Delta)^\frac{k-1}{2}h, & k=1,3,5,...,
\end{cases}
$$
and
$$
\beta_0=\beta_0(k,n)=
\begin{cases}
\frac{n}{\omega_{n-1}}\Big(\frac{\pi^\frac{n}{2}2^k\Gamma\big(\frac{k+1}{2}\big)}{\Gamma\big(\frac{n-k+1}{2}\big)}\Big)^\frac{n}{n-2},
& k=2,4,6,...\\
\frac{n}{\omega_{n-1}}\Big(\frac{\pi^\frac{n}{2}2^k\Gamma\big(\frac{k}{2}\big)}{\Gamma\big(\frac{n-k}{2}\big)}\Big)^\frac{n}{n-2},
& k=1,3,5,....
\end{cases}
$$
Now that each component $u^l$ of $u$ is in $W^{m,p}(\Omega)\cap L^{q}(\Omega)$. So the representation formula (\ref{e36}) for the even
orders can extend to $u^l$ via the density of
$C^\infty_0(\Omega)$ in $W^{m,p}(\Omega)$:
\begin{equation*}
u^l(x)=\begin{cases}
(-1)^\frac{m}{2}\Big(\frac{\omega_{n-1}\beta_0}{n}\Big)^\frac{2-n}{n}I_m(f_{m,l}), & m=2,4,6,...\\
(-1)^\frac{m-1}{2}\Big(\frac{\omega_{n-1}\beta_0}{n}\Big)^\frac{2-n}{n}I_m(f_{m,l}),
& m=1,3,5,....
\end{cases}
\end{equation*}
where
\begin{equation*}
f_{m,l}(y)=\begin{cases}\nabla^m u^l(y), & m=2,4,...\\
|x-y|^{-1}(x-y)\cdot\nabla^m u^l(y), & m=1,3,....\\
\end{cases}
\end{equation*}
For simplicity, set $f=f_{m,l}$ and $g=I_m(f_{m,l})$. An application of $|D^m u|\in L^{p,\lambda}(\Omega)$ implies $f_{m,l}\in
L^{p,\lambda}(\Omega)$. Now, we use $0<\epsilon$-mollifier $f_{\epsilon}=\phi_\epsilon\ast f$ of $f$ to obtain that if $t>0$ then
\begin{eqnarray*}
t^\frac1p&<&\liminf_{r\to 0} \left(\fint_{B_r(x_0)}\Big|g -\fint_{B_r(x_0)}g\Big|^p\right)^\frac1p\\
&\lesssim&\Big(\mathcal{M}\big(I_{m}(f-f_{\epsilon})\big)^p(x_0)\Big)^\frac1p\\
&\lesssim&I_m\big(\mathcal{M}(|f-f_\epsilon|^p)\big)^\frac1p.
\end{eqnarray*}
According to the definition of $C_m(\cdot; L^{p,\mu}(\Omega))$ and (\ref{e35}), we have that if
$$
\mathsf{S}_p(g,\Omega,t)=\left\{x\in\Omega:\ \limsup_{r\to
0}\fint_{B_r(x_0)}\Big|g-\fint_{B_r(x_0)}g\Big|^p>t\right\}.
$$
then
$$
C_m\big(\mathsf{S}_p(g,\Omega,t);L^{p,\mu}(\Omega)\big)\lesssim t^{-p}\|f-f_\epsilon\|_{L^{p,\mu}(\Omega)}^p.
$$
This last estimate, along with (\ref{e32}) of Lemma \ref{l1} ensuring
$$
\lim_{\epsilon\to 0}\|f-f_{\epsilon}\|_{L^{p,
\mu}(\Omega)}=0\quad\forall\quad\mu\in (\lambda,n],
$$
yields
\begin{equation}\label{e37}
  C_m\big(\mathsf{S}_p(g,\Omega); L^{p,\mu}(\Omega)\big)=0\quad{\forall}\quad \mu\in (p(m+n/q),n].
  \end{equation}

Thirdly, the $\mathsf{T}$-part of Theorem \ref{t4} is used to give
\begin{equation}\label{e38}
  C_m\big(\mathsf{T}(g,\Omega); L^{p,\mu}(\Omega)\big)=0\quad{\forall}\quad \mu\in (p(m+n/q),n].
  \end{equation}
Now, putting (\ref{e37}) and (\ref{e38}) together, we find
\begin{equation*}
  C_m\big(\Sigma_p(g,\Omega); L^{p,\mu}(\Omega)\big)=0\quad{\forall}\quad \mu\in (p(m+n/q),n].
  \end{equation*}
This plus Lemma \ref{l2} yields
$$
\Lambda_d^{(\infty)}(\Sigma_p(g,\Omega))=0\quad\forall\quad d>\epsilon+{np}/{q}\quad\&\quad\epsilon\in (0,1)
$$
thereby deriving $\hbox{dim}_H\big(\Sigma_p(g,\Omega)\big)\le np/q$, and so (\ref{e34}).

{\it Case 2}:\ $\hat{p}=1$. Under this assumption, (\ref{e34}) follows from the above argument and Theorem \ref{t4} with $\lambda=mp+np/q$, $\alpha=m$ and $p=p$.

\begin{remark} The second derivative estimates in the case $m=1$ follow \cite[Theorem 8.15]{BF2002}. Using the finite difference operator $\Delta_2 \phi(x)=\phi(x + z)-\phi(x)$, we
can prove that the second derivatives of the solution in Theorem \ref{t1} lie in the Morrey space $L^{p, 2m + np/q},$ at least when the
coefficients of our pde have derivatives, i.e., $a_{i x_j}^{k} =
\frac{\partial a_{i}^{k}}{\partial x_j} \in L^{2} $ and $
  a_{ij}^{kl} = \frac{\partial a_{i}^{k}}{\partial p_{j}^{l}}
  $ satisfies ellipticity and are bounded; see \cite{BF2002}. However, this Morrey space estimate does not decrease locally the size of the singular set; it only increases the dimension of the
underlying space $\mathbb R^{n}$ that would be needed to admit
singularities (non-H\"older solutions). Now we would need at least
$n\geq 5,$ to get such a solution.
\end{remark}

\subsection{Proof of Theorem \ref{t2}} In what follows, suppose $u$ is a $W^{1,2}\cap L^\infty$ minimizer of $\mathcal{J}(u,\Omega)$. Clearly, such $u$ is a $W^{1,2}\cap L^\infty$ minimizer of $\mathcal{J}(u,B)$ for any ball $B=B_r(x_0)\Subset\Omega$. For each $t\in (0,1)$ and $\tau>0$ let
$$
\Phi(t,\tau,r,x_0)=t^{2-n}e^{\tau t^\beta}\int_{B_{tr}(x_0)}A^{kl}_{ij}(x,u)u_{x_k}^iu_{x_l}^j.
$$
According to \cite{GG1984}, there is a constant $\tau$ (independent of $x_0$ and $t,r$) such that $t\mapsto\Phi(t,\tau,r,x_0)$ is an increasing function on the interval $(0,1)$. As a consequence, one has
$$
\Phi(t_1,\tau,r,x_0)\le \Phi(t_2,\tau,r,x_0)\quad\hbox{for}\quad 0<t_1<t_2<1.
$$
This, along with the elliptic condition on $A^{kl}_{ij}$, implies $|\nabla u|\in L^{2,2}(B)$. Of course, the argument for Theorem \ref{t1} derives $\hbox{dim}_H(\Sigma_2(u,B))=0$.

Next, suppose $\Sigma_2(u,\Omega)$ (which equals $\mathsf{S}_2(u,\Omega)$ since $u$ is bounded) is contained properly in a ball $B\subset\Omega$. Then $\hbox{dim}_H(\Sigma_2(u,\Omega))=0$ follows from the above argument. Without loss of generality we may assume that $B$ is just the unit ball $\mathbb B^n$. Since $t\mapsto\Phi(t,\tau,1,0)=:\Psi(t)$ is an increasing function on $(0,1)$, according to \cite[(15)]{GG1984} one has
\begin{equation}\label{e39}
\int_{\partial\mathbb B^n}|u(rx)-u(sx)|^2\,d{H}_{n-1}(x)\lesssim \frac{\Psi(r)-\Psi(s)}{\big(\ln\frac{r}{s}\big)^{-1}}\quad \ \forall\ 0<s<r<1.
\end{equation}
In the above and below, $dH_{n-1}$ stands for $n-1$ dimensional Hausdorff measure.

If $\{x_\rho\}_{\rho=1}^\infty$ is a sequence of points in $\Sigma_2(u,\Omega)$, then this sequence has a subsequence, still denoted by $\{x_\rho\}$, that converges to a point $x_0\in\mathbb B^n$ thanks to $\Sigma_2(u,\Omega)\Subset\mathbb B^n$. For simplicity, set $x_0$ be just the center of $\mathbb B^n$, and $r_\rho=2|x_\rho|<1$. Then, $\hat{u}(x)=u(r_\rho x)$ is a local minimizer of
$$
\mathcal{J}_\rho(\hat{u},\mathbb B^n)=\int_{\mathbb B^n} \hat{A}^{kl}_{ij}(x,\hat{u}) \hat{u}^i_{x_k} \hat{u}^j_{x_l}\quad\hbox{with}\quad \hat{A}^{kl}_{ij}(x,\hat{u})=A^{kl}_{ij}(r_\rho x,\hat{u}).
$$
Referring to the argument on \cite[Page 52]{GG1984}, $\{u(r_\rho x)\}$ converges weakly in $L^2(\mathbb B^n)$ to a map $v$ which is a local minimizer of
$$
\mathcal{J}_0(v,\mathbb B^n)=\int_{\mathbb B^n} {A}^{kl}_{ij}(0,v) {v}^i_{x_k} {v}^j_{x_l}.
$$
Moreover, there is a point $y_0$ with $|y_0|=1/2$ such that $y_0$ is a singular point of $v$.
Note that (\ref{e39}) is satisfied by $\hat{u}$, i.e.,
\begin{equation}\label{e310}
\int_{\partial\mathbb B^n}|\hat{u}(rx)-\hat{u}(sx)|^2\,d{H}_{n-1}(x)\lesssim \frac{\Psi({rr_\rho})-\Psi({sr_\rho})}{\big(\ln \frac{r}{s}\big)^{-1}}\quad\forall\ 0<s<r<1.
\end{equation}
Letting $\rho\to\infty$ in (\ref{e310}) produces
$$
\int_{\partial\mathbb B^n}|v(rx)-v(sx)|^2\,d{H}_{n-1}(x)=0.
$$
This indicates that $v$ is constant along the segment from the center of $\mathbb B^n$ to any point in $\partial\mathbb B^n$. Because $y_0$ is a singular point of $v$, one concludes that the segment between the center of $\mathbb B^n$ and $y_0$ is a subset of $\Sigma_2(v,\mathbb B^n)$, and so that $\hbox{dim}_H(\Sigma_2(v,\mathbb B^n))>0$, contradicting  $\hbox{dim}_H(\Sigma_2(v,\mathbb B^n))=0$ which follows from
$$
0\le\hbox{dim}_H(\Sigma_2(u,\mathbb B^n))\le\hbox{dim}_H(\Sigma_2(u,\Omega))=0.
$$
Therefore, $\Sigma_2(u,\Omega)$ consists of at most isolated points.

The part about the minimizing harmonic maps into a compact Riemannian manifold with metric $g$ is an immediate consequence of the above argument in that
$$
r\mapsto r^{2}\fint_{B_r(x_0)\Subset\Omega}|\nabla u|_g^2
$$
is increasing; see e.g. \cite{LS1996, LW2008}.

\begin{remark} In his survey paper \cite{LS1996}, Simon used the traditional ``blow-up" method
to show that if $u$ is a minimizing harmonic map from $\Omega$
(which is allowed to be unbounded) into $\mathbb S^2$ then the singular set of $u$ (possibly being a global singular set) has its Hausdorff dimension at most $n-3$. Especially, this singular set is
countably $(n-3)$-rectifiable. But note that rectifiability doesn't
include sets of Hausdorff measure zero, so one could in Simon's
situation have sets of fractional dimension, which by Theorem \ref{t2} one can not have.
\end{remark}

\subsection{Proof of Theorem \ref{t3}} The proof uses Simon's characterization of the so-called tangent maps associated with each point $y\in\hbox{sing}(u,\mathbb B^n)$; see also \cite{LS1996}.

Each tangent map can be found by passing to the limit in the energy norm (for some possible subsequence) of $\rho \to 0$ in
$u(y+\rho z)\to\phi_y(z)$. And then, it turns out, $\phi_y:\mathbb R^n\to\mathbb{S}^{m-1}$. Also, upon rotating $\phi_y$ one can write, as a limit,
\begin{equation} \notag
\phi_0 \! \left(\frac{\hat{x}}{|\hat{x}|}\right), \, \hat{x} \in \mathbb R^d, \, \text{for some $d \geq 3$}.
\end{equation}
Here $x = (\hat{x}, \bar{x})$, $\bar{x} \in \mathbb R^{n-d}$. $\phi_0(\hat{\xi})$ is smooth on $|\hat{\xi}| \leq 1$. Upon considering
$$
I_\rho(z,y)=\int |z-x|^{1-n} \, |D(u(y+\rho x))| \, dx,
$$
we find two ways to evaluate the limit of $I_\rho(z,y)$ as $\rho\to 0$: on the one hand,
\begin{eqnarray*}
\lim_{\rho\to 0}I_\rho(z,y)&=& \lim_{\rho\to 0}\int |z-x|^{1-n} \, \rho \, |(Du)(y + \rho x)| \, dx\\
&=& \lim_{\rho\to 0}\int |\rho z - w + y|^{1-n} \, |(Du)(w)| \, dw=I_1(|Du|)(y);\notag
\end{eqnarray*}
on the other hand,
\begin{eqnarray*}
\lim_{\rho\to 0}I_\rho(z,y)&=&\int \int_{|\hat{x}| \leq 1} \left(|\hat{z} - \hat{x}|^2 + |\bar{z} - \bar{x}|^2\right)^{(1-n)/2} \, \left|D\phi_0 \left(\frac{\hat{x}}{|\hat{x}|}\right)\right| \, d\hat{x} \, d\bar{x}\\
&=& c \int_{|\hat{x}| \leq 1} |\hat{z} - \hat{x}|^{1-d} \, \left|(D\phi_0)\left(\frac{\hat{x}}{|\hat{x}|}\right)\right| |\hat{x}|^{-1} \, d\hat{x}
\end{eqnarray*}
with $c$ being a constant. Thus, $I_1(|Du|)(y)$ diverges at $\hat{z}=0$ which corresponds to $y \in\hbox{sing} (u,\mathbb B^n)$. We can clearly repeat this for any $y\in \hbox{sing}(u,\mathbb B^n)$.  Thus (\ref{eqF}) holds because $I_1(|Du|)$ is smooth otherwise, due to the known smoothness of $|Du|$ off the $\hbox{sing}(u,\mathbb B^n)$.

\begin{remark} In accordance with \cite[page 105]{G1983} and \cite[Corollary 2.2.8]{LW2008} one has that if $u$ is a $W^{1,2}$-minimizing harmonic map from $\mathbb B^n$ into $\mathbb S^{m-1}$ then
$$
\mathsf{S}_2(u,\mathbb B^n)=\hbox{sing}(u,\mathbb B^n)=\left\{x_0\in\mathbb B^n:\ \lim_{r\to 0}r^2\fint_{B_r(x_0)\Subset\mathbb B^n}|\nabla u|^2>0\right\}.
$$
On the other hand, for $n\ge 4$ let $u(x,y)=x/|x|: \mathbb
R^3\times\mathbb R^{n-3}\mapsto \mathbb S^2$. According to
\cite[Page 16]{LW2008}, this is a minimizing harmonic map. It is
not difficult to see that as a global singular set,
$$
\hbox{sing}(u,\mathbb R^n)=\{0\}\times \mathbb
R^{n-3}=\mathsf{S}_2(u,\mathbb R^n)\ \ \hbox{where}\ \ I_1(|\nabla
u|)(x,y)=\infty,
$$
and so $\hbox{dim}_H\big(\hbox{sing}(u,\mathbb R^n)\big)=n-3.$
It is worth noticing that one cannot get a local singular set of the
type $\{0\}\times B_R(0)$ for some ball $B_R(0)$ in $\mathbb
R^{n-3}$ by cutting off $u(x,y)$ and then passing to a limit. In
fact, consider
$$
u_j(x,y)=\Big(\frac{x}{|x|}\Big)\phi_j(y)\quad\forall\ \
(x,y)\in\mathbb R^3\times\mathbb R^{n-3}\quad\&\quad j=1,2,3,...,
$$
where
$$
\phi_j(y)=\phi_j(|y|)=\begin{cases} 1\quad &\quad \hbox{for}\quad |y|\le 1-\frac1j\\
\hbox{linear}\quad &\quad \hbox{for}\quad 1-\frac1j\le |y|\le 1\\
0\quad &\quad \hbox{for}\quad |y|\ge 1.
\end{cases}
$$
Of course, this function $\phi_j$ is only Lipchitz, but can be made
better if needed -- the conclusion is the same. Since $|\nabla
\phi_j|\approx j$, one concludes
$$
\int_{x\in\mathbb B^3}\int_{y\in\mathbb B^{n-3}}|\nabla
u_j(x,y)|^2\approx \int_{y\in\mathbb B^{n-3}}\int_{x\in\mathbb
B^3}\big(|x|^{-2}+|\nabla\phi_j(y)|^2\big)\to\infty\ \ \hbox{as}\ \
j\to\infty.
$$
Note that $\lim_{j\to\infty}u_j(x,y)$ ought to be $u(x,y)$ (which is
a minimizing harmonic map) on $\mathbb B^3\times\mathbb B^{n-3}$, the
approximation of the constant function is no good in $W^{1,2}$, and
$u_j$ is not a minimizing harmonic map. So, it is impossible to ``bite off" a
piece of a global singular set and to get a local singular set.
\end{remark}


\begin{thebibliography}{99}

\bibitem{Ad} D. R. Adams, {\it A note on Choquet integral with respect to
Hausdorff capacity}, {in ``Function Spaces and Applications,}" Lund
1986. {Lecture Notes in Math.} {1302}, Springer-Verlag, 1988,
pp. 115-124.

\bibitem{AAnn} D. R. Adams, {\it A sharp inequality of J. Moser for higher order derivatives},
{Ann. Math.} 128(1988)385-398.

\bibitem{AH1996} D. R. Adams and L. I. Hedberg, \emph{Function Spaces and Potential Theory}. {Springer-Verlag}, Berlin Heidelberg, 1996.

\bibitem{AX2004} D. R. Adams and J. Xiao, {\it Nonlinear analysis on Morrey spaces and their capacities}, {Indiana Univ. Math. J.} {53}(2004)1629-1663.

\bibitem{AX2011b} D. R. Adams and J. Xiao, {\it Morrey potentials and harmonic maps}, {Comm. Math. Phys.} 308 (2011)439-456.

\bibitem{AX2011a} D. R. Adams and J. Xiao, {\it Morrey spaces in harmonic analysis}, {Ark. Mat.} 50(2012)201-230.

\bibitem{AX2011c} D. R. Adams and J. Xiao, {\it Regularity of Morrey commutators}, {Trans. Amer. Math. Soc.} 364(2012)4801-4818.

\bibitem {BF2002} A. Bensoussan and J. Frehse, \emph{Regularity results for
nonlinear elliptic systems and applications,} volume 151 of
{Applied Mathematical Sciences}. Springer-Verlag, Berlin, 2002.

\bibitem {CF1988} F. Chiarenza and M. Frasca, {\it Morrey spaces and Hardy-Littlewood
maximal function}, {Rend. Mat. Appl. (7)}, 7(3-4)(1988)(1987)273-279.

\bibitem{DG1968} E. De Giorgi, {\it Un esempio di estremali discontinue per un problema variazionale di tipo ellittico}, Boll. UMI 4(1968)135-137.

\bibitem{F1969} A. Friedman, \emph{Partial differential equations}, Holt,
Rinehart and Winston, Inc., New York, 1969.

\bibitem {G1973} F. W. Gehring, {\it The {$L^{p}$}-integrability of the partial
derivatives of a quasiconformal mapping}, {Acta Math.} 130(1973)265-277.

\bibitem{G1983} M. Giaquinta, {\it Multiple Integrals in the Calculus of
Variations and Nonlinear Elliptic Systems}, {Ann. Math. Studies}
105, Princeton University Press, Princeton, N.J., 1983.

\bibitem{GG1982} M. Giaquinta and E. Giusti, {\it On the regularity of the minima of variational integrals}, {Acta Math.} 148(1982)31-46.

\bibitem{GG1984} M. Giaquinta and E. Giusti, {\it The singular set of the minima of certain quadratic functionals,} {Annali della Scuola Normale Superiore di Pisa} 11:1(1984)45-55.

\bibitem{GM1968} E. Giusti and M. Miranda, {\it Sulla regolarit\`a delle soluzioni deboli di una classe di sistemi ellittici quasilineari}, Arch. Rat. Mech. Anal. 31(1968)173-184.

\bibitem{HL1990} R.~Hardt and F.~Lin,
{\it The singular set of an energy minimizing map from $B^4$ to $S^2$},
{Manuscripta Math.} 69(1990)275-289.

\bibitem{K1995} A. Koshelev, {\it Regularity problem for quasilinear elliptic and parabolic systems}, Lecture Notes in Mathematics, 1614. Springer-Verlag, Berlin, 1995. xxii+255 pp.

\bibitem{LW2008} F.~Lin and C.~Wang.
\emph{The analysis of harmonic maps and their heat flows.}
World Scientific Publishing Co.~Pte.~Ltd., Hackensack, NJ, 2008.

\bibitem{M1963} N. G. Meyers, {\it An {$L^{p}$} estimate for the gradient of
solutions of second order elliptic divergence equations}, {Ann.
Scuola Norm. Sup. Pisa (3)} 17(1963)189-206.

\bibitem{ME1975} N. G. Meyers and A. Elcrat, {\it Some results on regularity for
solutions of non-linear elliptic systems and quasi-regular
functions}, {Duke Math. J.} 42(1975)121-136.

\bibitem{M1938} C. B. Morrey, {\it On the solutions of quasi-linear elliptic partial differential equations}, Trans. Amer. Math. Soc. 43(1938)126-166.

\bibitem{SY1988} R. Schoen and S. T. Yau, {\it Conformally flat manifolds, Kleinian groups and scalar curvature}, Invent. Math. 92(1988)47-71.

\bibitem{LS1996} L. Simon, {\it Singularities of geometric variational problems},
In {Nonlinear partial differential equations in differential geometry (Park City, UT, 1992)}, R.~Hardt and M.~Wolf ed., vol.~2 of \emph{IAS/Park City Math.~Series}, 185-223.  Amer.~Math.~Soc., Providence, RI, 1996.


\bibitem{Z1986} C. T. Zorko, {\it Morrey spaces}, {Proc. Amer. Math. Soc.} 98(1986)586-592.

\end{thebibliography}
\end{document}